\documentclass[12pt,letterpaper]{amsart}

\usepackage[top=1in, bottom=1in, left=1in, right=1in]{geometry}

\usepackage[utf8]{inputenc}
\usepackage{amssymb,amsthm,amsmath}
\usepackage{enumitem}
\usepackage{mathtools}
\usepackage{microtype}
\usepackage{times}
\usepackage{dsfont}
\usepackage{mathrsfs}
\usepackage[dvipsnames]{xcolor}

\PassOptionsToPackage{hyphens}{url}\usepackage[colorlinks=true, linkcolor=NavyBlue, citecolor=teal, urlcolor=gray]{hyperref}

\newcommand{\NN}{\mathbb{N}}
\newcommand{\ZZ}{\mathbb{Z}}
\newcommand{\RR}{\mathbb{R}}
\newcommand{\PP}{\mathbb{P}}
\newcommand{\EE}{\mathbb{E}}

\newcommand{\pextra}{P_{\mathtt{extra}}}
\newcommand{\dtv}{\mathrm{d}_{\mathrm{TV}}}
\newcommand{\bolddelta}{\boldsymbol{\delta}}
\newcommand{\dee}{\,\mathrm{d}}
\newcommand{\one}{\mathds{1}}
\newcommand{\RS}{\mathscr{R}}
\renewcommand{\le}{\leqslant}
\renewcommand{\ge}{\geqslant}
\renewcommand{\epsilon}{\varepsilon}

\theoremstyle{plain}
\newtheorem{cor}{Corollary}[section]
\newtheorem{lemma}[cor]{Lemma}
\newtheorem{prop}[cor]{Proposition}
\newtheorem{theorem}{Theorem}
\newtheorem{mainlemma}{Main Lemma}

\theoremstyle{definition}
\newtheorem{deff}[cor]{Definition}

\theoremstyle{remark}
\newtheorem*{remss}{Remarks}

\usepackage{bookmark}

\newcommand{\thetasmall}{\Theta_{\infty}^{(1)}}
\newcommand{\thetalarge}{\Theta_x^{(2)}}

\newcommand{\floor}[1]{\lfloor #1 \rfloor}

\newcommand{\primes}{\mathbf{Primes}}
\newcommand{\pahk}{\mathbb{P}_\star}
\newcommand{\eahk}{\mathbb{E}_\star}

\newcommand{\ymin}{y_{\min}}
\newcommand{\zmin}{z_{\min}}
\newcommand{\umin}{u_{\min}}
\newcommand{\vmin}{v_{\min}}

\newcommand{\bg}{\big}
\newcommand{\bgg}{\Big}
\newcommand{\bggg}{\bigg}

\title[Poisson--Dirichlet approximation]{Poisson--Dirichlet approximation for counting\\integers with divisors in an interval}

\author{Tony Haddad}
\address{Department of Mathematics and Statistics\\
	University of Turku\\
	20014 Turku\\
	Finland}
\email{{\tt tohadd@utu.fi}}

\date{\today}

\begin{document}

\begin{abstract}
    We give a simple inequality that compares the laws of two random variables taking values in a convex subset of a normed vector space. By combining this with Arratia's coupling, recently refined by Koukoulopoulos and the author, we obtain a general strategy to reduce the problem of finding an asymptotic formula for the number of integers whose prime factorization lies in any given subset of $\ell^1(\mathbb R)$ to bounding two key probabilities measuring proximity to the boundary of the subset in question.  

    We apply this strategy to obtain an asymptotic formula for counting integers in $[1, x]$ that have a divisor in an interval $(y, z)$ in the regime $z/y \to \infty$ as $x \to \infty$.
\end{abstract}

\maketitle

\section{Introduction}

    Let $H(x, y, z)$ be the number of integers in $[1, x]$ having a divisor in the open interval $(y, z)$. The size of this function was studied in various ranges over the last century. In 1934, Besicovitch \cite{besicovitch1934} showed that
    \[
        \liminf_{y \to \infty} \lim_{x \to \infty} \frac{H(x, y, z)}{x} = 0
    \]
    when $z = 2y$. In 1935, Erd\H os \cite{erdos1935} replaced the $\liminf$ by $\lim$ in the statement above, and he later showed in \cite{erdos1936} a similar statement for any $z = z(y)$ which satisfies $\frac{\log z/y}{\log y} \to 0$ as $y \to \infty$. Erd\H os later proved in 1960 \cite{erdos1960} that 
    \[
        \lim_{x \to \infty} \frac{H(x, y, 2y)}{x}= (\log y)^{-\delta + o(1)}
    \]
    with $\delta \coloneqq 1-\frac{1+\log_2 2}{\log 2} = 0.086\ldots$ as $y \to \infty$. 
    
    In 1980, Tenenbaum \cite{tenenbaum1980} proved that the following limit exists for all $0\le u<v\le 1$:
	\begin{equation}
    \label{eq:tenenbaum limit result}
	h(u, v) \coloneqq \lim_{x \to \infty} \frac{H(x, x^{u}, x^v)}{x}.
	\end{equation} 
    In 1984, Tenenbaum \cite{tenenbaum1984} gave precise upper and lower bounds for $H(x, y, z)$ in various ranges of $(x, y, z)$, and in 2008, Ford \cite{ford2008} was able to determine the order of magnitude of $H(x, y, z)$ for all $(x, y, z)$ satisfying $2 \le y < z \le x$. One of the key estimates obtained in \cite{ford2008} is the following:
	\begin{equation}
		\label{eq:ford result}
		H(x, y, z) \asymp xw^{\delta}(\log 2/w)^{-3/2}
	\end{equation}
	whenever $2y \le z \le \min\{y^2, x\}$, $100 \le y \le x^{2/3}$ and $x$ is sufficiently large, where $w \coloneqq \tfrac{\log z/y}{\log y}$. 
    
    The main theorem of this paper gives the range of $(x, u, v)$ for which the density $h(u, v)$ approximates well the function $H(x, x^u, x^v)/x$:

    \begin{theorem}
        \label{theo:asymp for hxyz}
        Let $\delta \coloneqq 1 - \frac{1+\log_2 2}{\log 2}$. We have
        \[
            H(x, y, z) = x \cdot h\bg(\tfrac{\log y}{\log x}, \tfrac{\log z}{\log x}\bg) + O\bggg(\frac{x}{(\log \ymin)^{\delta}(\log_2 \ymin)^{3/2}}\bggg)
        \] 
        for all $3 \le y < z \le x/3$, with $\ymin \coloneqq \min\{y, x/z\}$.
    \end{theorem}    

    By using \eqref{eq:ford result}, this theorem yields an asymptotic when $x\to \infty$, provided that $y$ and $z$ are functions of $x$ satisfying $z/y \to \infty$.

    In the same paper where Tenenbaum proved that the limit $h(u, v)$ exists \cite{tenenbaum1980}, he also gave an explicit formula for it in the range $(u, v) \in [0, \frac{2}{5}] \times \{\frac{1}{2}\}$. We give an explicit formula for $h(u, v)$ for any $u, v \in [0, 1]$ satisfying $u < v$:
    
    \begin{theorem}
    \label{thm:explicit formula for h(u,v)}
        For any fixed $u, v \in [0, 1]$ satisfying $u < v$, let $d \coloneqq \floor{\frac{1}{v-u}}$, and let $\Delta^d$ be the simplex $\{(x_1, \ldots, x_d) \in \RR^d : x_1 \ge \cdots \ge x_d \ge 0 \text{ and } \sum_i x_i \le 1\}$. Let $\rho \colon [0,\infty) \to \RR$ be the Dickman function defined by the delay differential equation $u\rho'(u)=-\rho(u-1)$ for $u > 1$, and the initial condition $\rho([0, 1]) = \{1\}$. For any family $\mathcal E$ in the power set $2^{2^{\NN\cap[1, d]}}$, we also define the (possibly empty) compact convex polytope $P_{u, v}(\mathcal E)$ as
        \begin{align*}
            P_{u, v}(\mathcal E) \coloneqq \bggg\{\mathbf x \in \Delta^d : \ &\sum_{\substack{1 \le i \le d \\ i \in E}} x_i \ge v \text{ for all $E \in \mathcal E$,} \text{ and } \sum_{\substack{1 \le i \le d \\ i \notin E}} x_i \ge 1-u \text{ for all $E \in 2^{\NN \cap [1, d]} \setminus \mathcal E$}\bggg\}.
        \end{align*}
        Then we have 
        \[
            h(u, v) = 1- \sum_{\mathcal E \in 2^{2^{\NN\cap [1, d]}}} \underset{(t_1, \ldots, t_d) \in P_{u, v}(\mathcal E)}{\int\dots\int} \rho\left(\frac{1-\sum_{i=1}^d t_i}{t_d}\right)\frac{\dee t_1 \cdots \dee t_d}{t_1 \cdots t_d}
        \]
        for every $u, v\in [0, 1]$ with $u < v$.
    \end{theorem}

    It is not a coincidence that the integral in this theorem involves the joint density of the first $d$ terms of a Poisson--Dirichlet process. In fact, the function $H(x, y, z)$ counts integers with some constraints on the size of all its prime factors, and as we will see, the Poisson--Dirichlet process is a good model for the prime factorization of a random integer.

    At the end of the next section, we reduce Theorems \ref{theo:asymp for hxyz} and \ref{thm:explicit formula for h(u,v)} to a few key lemmas. We will first explain a general strategy to make this reduction possible for a family of problems which includes the theorems above.

    \subsection*{Notation}
    In some parts of the paper, we will use the unique factorization of positive integers $n=n^\flat n^\sharp$ with $n^\flat$ being squarefree, $n^\sharp$ being squarefull, and $\gcd(n^\flat, n^\sharp) = 1$. The variable $p$ is always reserved for prime numbers unless stated otherwise. We let $\log_j$ denote the $j$-iteration of the natural logarithm, meaning that $\log_1 = \log$ and $\log_j = \log \circ \log_{j-1}$ for $j \ge 2$. The norm $\|\cdot\|$ will always be the $\ell^1$-norm, except in the statement and proof of Proposition \ref{prop:fundamental ineq} which considers $\|\cdot\|$ to be any arbitrary norm. If $P$ is some proposition, then the indicator $\one_P$ will be equal to $1$ if $P$ is true, and $0$ if $P$ is false. To describe various estimates, we use Vinogradov's notation $f(x) \ll g(x)$ or Landau's notation $f(x) = O(g(x))$ to mean that $|f(x)| \le C \cdot g(x)$ for a positive constant $C$. If $C$ depends on a parameter $\alpha$, we write $f(x) \ll_\alpha g(x)$ or $f(x) = O_\alpha(g(x))$.

    \subsection*{Acknowledgments}
    The author is grateful to Kevin Ford, Dimitris Koukoulopoulos, Kaisa Matomäki and Gérald Tenenbaum for helpful comments on this project. The author was supported by the Courtois Chair II in fundamental research at the start of the project, and is now supported by the Research Council of Finland grants nos. 364214 and 370133.

\section{Reduction to boundary lemmas}

    \subsection{Theorems \ref{theo:asymp for hxyz} and \ref{thm:explicit formula for h(u,v)} under a general framework}
    
    Let $\ell^1(\RR)$ take its usual definition: It is the Banach space of real sequences $\mathbf x = (x_i)_{i \ge 1}$ with the property that $\sum_{i \ge 1} |x_i| < \infty$. We equip it with the $\ell^1$-norm defined by $\|\mathbf x\| \coloneqq \sum_{i \ge 1} |x_i|$. We define $\Delta$ to be the set of sequences $\mathbf x \in \ell^1(\RR)$ satisfying both $x_1 \ge x_2 \ge \cdots \ge 0$ and $\|\mathbf x\| \le 1$.

    For any positive integer $n$, we set $(p_i(n))_{i \ge 1}$ to be the unique sequence satisfying the factorization $n=\prod_{i = 1}^{\infty} p_i(n)$, with $p_1(n) \ge p_2(n) \ge \cdots$ all being primes or ones. The following sequence of relative sizes of prime factors 
    \[
        \primes(n, x) \coloneqq \left(\frac{\log p_i(n)}{\log x}\right)_{i \ge 1}
    \]
    is in the set $\Delta$ for any $n, x$ satisfying $1 \le n \le x$. We assign two probability measures on the measurable space $(\Delta, \mathcal B(\Delta))$ with $\mathcal B(\Delta)$ being the Borel $\sigma$-algebra of $\Delta$: For any fixed $x \ge 2$, let $\nu_x$ be the probability measure defined by 
        \[ 
            \nu_x(A) \coloneqq \frac{1}{\floor x}\cdot \#\bgg\{1 \le n \le x : \primes(n, x) \in A\bgg\}.
        \]
        for any set $A \in \mathcal B(\Delta)$. Let $\mu$ be the Poisson--Dirichlet distribution of parameter $1$. One way to define this probability measure is with its finite-dimensional distributions\footnote{Note that finite-dimensional distributions of stochastic processes allow us to determine probabilities of events in cylindrical $\sigma$-algebras. In the case of $\ell^1(\RR)$, any closed balls of the form $\overline B(\mathbf y, r) \coloneqq \{\mathbf x \in \ell^1(\RR) : \|\mathbf x - \mathbf y\| \le r\}$ for some $\mathbf y \in \ell^1(\RR)$ and $r > 0$ can be generated by the countable intersection $\bigcap_{k=1}^\infty E_k(\mathbf y, r)$ with $E_k(\mathbf y, r)$ being the cylindrical set of real sequences $(x_i)_{i \ge 1}$ satisfying $\sum_{i=1}^k |x_i-y_i| \le r$. Since $\ell^1(\RR)$ is a separable space, any open set of $\ell^1(\RR)$ can be generated by closed balls, and therefore finite-dimensional distributions uniquely determine the law $\mu$ on $\mathcal B(\Delta)$.}:
        \begin{equation}
        \label{eq:density of PD}
            \mu(Q_k(\mathbf u)) \coloneqq \int_{u_k}^\infty \cdots \int_{u_1}^\infty \frac{\one_{0 \le t_k \le \cdots \le t_1} \cdot \one_{t_1 + \cdots + t_k \le 1}}{t_1 \cdots t_k}\rho\left(\frac{1-\sum_{i=1}^k t_i}{t_k}\right)\ \dee t_1 \cdots \dee t_k
        \end{equation}
        for all $k \ge 1$, $\mathbf u = (u_1, \ldots, u_k) \in \RR^k$ and with $Q_k(\mathbf u)$ being the set of real sequences $(x_i)_{i \ge 1}$ with $x_i \ge u_i$ for all $1 \le i \le k$. In the above definition, the function $\rho \colon [0,\infty) \to \RR$ is again the Dickman function as defined in Theorem~\ref{thm:explicit formula for h(u,v)}. It is important to keep in mind that this measure is supported on $\Delta^* \coloneqq \{\mathbf x \in \Delta : \|\mathbf x\| = 1\}$. Other equivalent definitions of $\mu$ are given in \cite[Chapter~2]{feng2010}. 

        For real numbers $u < v$, we consider the set $D(u, v)$ of sequences $\mathbf x \in \Delta$ that have a subsum in the open interval $(u, v)$. It becomes clear that this set is open in $\Delta$ if we rewrite it as a union of the preimage of continuous functions over open intervals:
    \begin{equation}
    \label{eq:D(u,v) union def}
        D(u, v) = \bigcup_{E \subseteq \NN} \{\mathbf x \in \Delta : \mathbf x(E) \in (u, v)\},
    \end{equation}
    with the notation $\mathbf x(E) \coloneqq \sum_{i \in E} x_i$. By the definition of the measure $\nu_x$, we have
    \begin{equation}
        \label{eq:H(x,y,z) in prob language}
        H(x, y, z) = \floor x \cdot \nu_x\bgg(D\bg(\tfrac{\log y}{\log x}, \tfrac{\log z}{\log x}\bg)\bgg),
    \end{equation}
    and, for $0 \le u < v \le 1$, as we will prove at the end of the section, we have
    \begin{equation}
        \label{eq:h(u,v) is PD}
        h(u, v) = \mu\bg(D(u, v)\bg).
    \end{equation}
    We will present a technique to bound $|\nu_x(A) - \mu(A)|$ for any Borel set $A$.
    
    \subsection{Comparing laws of random variables on normed vector spaces}
    
    Before moving on to approximating $\nu_x(A)$, we first present a simple inequality which will be fundamental in our approach:

    \begin{prop}
    \label{prop:fundamental ineq}
        Let $\mathcal W$ be a convex set in any normed vector space. For any subset $A \subseteq \mathcal W$, let $\partial A$ be its boundary with respect to the subspace topology of $\mathcal W$. For any $\eta \ge 0$, we write $\|w-\partial A\| \coloneqq \inf_{z \in \partial A} \|w-z\|$ and
        \[
            \partial_\eta A \coloneqq \bgg\{w \in \mathcal W : \|w - \partial A\| < \eta\bgg\}.
        \]
        Let $(\Omega, \mathcal F, \PP)$ be any probability space. For any three random variables $X, Y, R$ on this space with both $X$ and $Y$ taking values in $\mathcal W$, and with $R$ taking values in $\RR_{\ge 0}$, we have
        \begin{align*}
            \bggg|\PP\bg[X \in A\bg] - \PP\bg[Y \in A\bg]\bggg| \le \PP\bgg[\{X, Y\} \subseteq \partial_R A\bgg] + \PP\bgg[\|X-Y\|\ge R\bgg],
        \end{align*}
        for any Borel set $A \subseteq \mathcal W$.
    \end{prop}

    \begin{proof}
    First, we have 
    \begin{align*}
        \bgg|\PP\bg[X \in A\bg] - \PP\bg[Y \in A\bg]\bgg| &= \bgg|\PP\bg[X \in A, Y \notin A\bg] - \PP\bg[X \notin A, Y \in A\bg]\bgg| \\
        &\le \max\bgg\{\PP\bg[X \in A, Y \notin A\bg],\ \PP\bg[X \notin A, Y \in A\bg]\bgg\}.
    \end{align*}
    For any two points $a, b \in \mathcal W$, let $\mathfrak L (a, b)$ be the line segment between $a$ and $b$. Suppose that exactly one point of $\{a, b\}$ belongs to $A$, then $\mathfrak L(a, b)$ intersects $\partial A$ since line segments are always connected. Thus, there must exist $c \in \mathfrak L(a, b) \cap \partial A$ that satisfies the inequality
    \[
        \|a - \partial A\| + \|b - \partial A\| \le \|a - c\| + \|c - b\| = \|a - b\|,
    \]
    which implies
    \[
        \bgg|\PP\bg[X \in A\bg] - \PP\bg[Y \in A\bg]\bgg| \le \PP\bgg[\|X - \partial A\| + \|Y - \partial A\| \le \|X - Y\|\bgg].
    \]
    The proposition follows directly from this.
    \end{proof}

    \subsection{Approximating prime factorizations by a Poisson--Dirichlet process}

    In 1972, Billingsley \cite{billingsley1972} proved that for any fixed $\mathbf u = (u_1, \ldots, u_k) \in [0, 1]^k$, we have
    \[
        \nu_x\bg(\Delta \cap Q_k(\mathbf u)\bg) = \mu\bg(\Delta \cap Q_k(\mathbf u)\bg) + o(1)
    \] 
    as $x \to \infty$, where $Q_k(\mathbf u)$ is again the set of real sequences $(x_i)_{i \ge 1}$ with $x_i \ge u_i$ for all $1 \le i \le k$. In other words, the $k$ largest prime factors of a random integer uniformly distributed in $\NN \cap[1, x]$ converge in distribution to the first $k$ components of a Poisson--Dirichlet process of parameter $1$. Later, research has been done on understanding the speed of convergence in Billingsley's Theorem. In 1976, Knuth and Trabb Pardo \cite{knuthtrabbpardo1976} studied the case where we fix $u_1 = \cdots = u_{k-1} = 0$ and $u_k \in [0, 1]$ (corresponding to the distribution of the $k$\textsuperscript{th} largest prime factor). In 2000, Tenenbaum \cite{tenenbaum2000} gave the following asymptotic expansion for the error term in Billingsley's Theorem: There exists a sequence of functions $(\varphi_j)_{j \ge 1}$ such that
    \begin{equation}
    \label{eq:tenenbaum quant billingsley}
        \nu_x\bg(\Delta \cap Q_k(\mathbf u)\bg) = \mu\bg(\Delta \cap Q_k(\mathbf u)\bg) + \sum_{j=1}^m \frac{\varphi_j(\mathbf u)}{(\log x)^j} + O_{\mathbf u, m}\bggg(\frac{1}{(\log x)^{m+1}}\bggg)
    \end{equation}
    for all fixed $m$ and $\mathbf u \in [0, 1]^k$.\footnote{Actually, the dependence in $\mathbf u$ in the implicit constant of the error term was given in \cite{tenenbaum2000}. It is however more complicated to state in one line.}
    
    In 2002, Arratia approached from a different angle the quantification of Billingsley's Theorem. He proved in \cite{arratia02} that there exists a coupling between a random integer $N_x$ uniformly distributed in $\NN \cap [1, x]$, and a Poisson--Dirichlet process $\mathbf V = (V_1, V_2, \ldots)$ of parameter $1$ such that the $\ell^1$-norm $\|\primes(N_x, x) - \mathbf V\|$ is in expectation $O(\frac{\log_2 x}{\log x})$. In 2025, Koukoulopoulos and the author \cite{hk2025} modified his coupling to improve this expectation to $O(\frac{1}{\log x})$. Note that in these couplings, the law of $\primes(N_x, x)$ is $\nu_x$ and the law of $\mathbf V$ is $\mu$.
    
    In \cite{hk2025}, we actually gave the following stronger result: Let $(\Omega_{\star}, \mathcal F_{\star}, \pahk)$ be the coupling constructed in \cite{hk2025}. We factorize $N_x$ into its unique factorization $N_x = N_x^\flat N_x^\sharp$, with $N^\flat_x$ being squarefree, with $N_x^\sharp$ being squarefull and $\gcd(N_x^\flat, N_x^\sharp) = 1$. We also consider the deterministic function $r\colon \RR_{> 0} \to \RR_{\ge 0}$ defined in \cite[Equation (2.3)]{hk2025} (we only need to keep in mind for our discussion that it is a function satisfying the bound $r(t) \ll \min\{t, t^{-2}\}$ for all $t > 0$). As a direct consequence of \cite[Lemma 2.1 and Proposition 2.4]{hk2025}, we deduce 
    \begin{equation}
    \label{eq:AHK coupling bound}
        \pahk\bgg[\bg\| \primes(N_x, x) - \mathbf V\bg\| \ge \max\{S_x, T_x\} \bgg] \ll \frac{1}{\log x},
    \end{equation}
    with 
    \begin{equation}
        \label{eq:def of Sx and Tx}
        S_x \coloneqq \frac{5}{\log x}\log(x/N_x^\flat) \qquad \text{and} \qquad T_x \coloneqq \frac{5}{\log x}\sum_{i \ge 1} r(V_i \log x).
    \end{equation}
    By applying Proposition \ref{prop:fundamental ineq} with \eqref{eq:AHK coupling bound}, we directly obtain the following theorem:

    \begin{theorem}
    \label{thm:ineq under hyp}
        Let $x \ge 2$. We have
        \begin{align*}
            \bgg|\nu_x(A) - \mu(A)\bgg| &\ll \pahk\bgg[\bg\{\primes(N_x, x),\ \mathbf V\bg\}\subseteq \partial_{\max\{S_x, T_x\}} A\bgg]+ \frac{1}{\log x},
        \end{align*}
        for any $A \in \mathcal B(\Delta)$.
    \end{theorem}

    As described in \cite[Lemmas 2.2 and 2.3]{hk2025}, the random variables $S_x$ and $T_x$ are both concentrated around $\frac{1}{\log x}$. Thus, we should expect the following from this theorem: 
    \begin{equation}
    \label{eq:heuristic for ineq under hypo}
        \text{RHS of Theorem \ref{thm:ineq under hyp}} \approx \min\bg\{\nu_x(\partial_\eta A), \mu(\partial_\eta A)\bg\} + \frac{1}{\log x}
    \end{equation}
    with $\eta = \frac{1}{\log x}$. This heuristic only becomes a problem if either $S_x$ or $T_x$ is positively correlated with the event that $\primes(N_x, x)$ and $\mathbf V$ are close to the boundary of $A$.

    However, if we do not have any clues on how to compute the probability on the right-hand side of Theorem \ref{thm:ineq under hyp}, then we can always use the following inequality
    \[
        \pahk\bgg[\|\primes(N_x, x) - \mathbf V\| \ge \eta\bgg] \ll_\epsilon x^{-\frac{\eta}{4+\epsilon}} + \frac{1}{\log x}.
    \]
    for any $\epsilon, \eta > 0$ and $x \ge 2$, which can be directly deduced from Chernoff's inequality with \cite[Propositions 2.4 and 2.5]{hk2025}. Therefore, we immediately arrive at the following theorem by applying again Proposition \ref{prop:fundamental ineq}. 

    \begin{theorem}
    \label{thm:ineq unconditional}
        Fix $\epsilon > 0$. For $x \ge 2$ and $\frac{1}{\log x} \le \eta \le \frac{(4+\epsilon)\log_2 x}{\log x}$, we have 
        \[
            \bgg|\nu_x(A) - \mu(A)\bgg| \ll_\epsilon \min\bgg\{\nu_x(\partial_\eta A),\ \mu(\partial_\eta A)\bgg\} + x^{-\frac{\eta}{4+\epsilon}},
        \]
        for any set $A \in \mathcal B(\Delta)$.
    \end{theorem}

    For most choices of set $A$, Theorem \ref{thm:ineq unconditional} leads to a weaker result than Theorem \ref{thm:ineq under hyp}, as discussed in \eqref{eq:heuristic for ineq under hypo}. However, the ratio between the bounds obtained in those two theorems should be relatively small, as we will see later in \eqref{eq:weaker H(x,y,z) estimate} when applying this theorem to the set $D(u, v)$. 

    \subsection{Reducing Poisson--Dirichlet approximation problems to two lemmas}

    All we need to get an appropriate bound using Theorem \ref{thm:ineq under hyp} are the following two lemmas:
    \begin{itemize}
        \item A \emph{Number Theory Boundary Lemma} (NTBL), which is an upper bound on the probability that either $\primes(N_x, x)$ or $\mathbf V$ is close to the boundary of $A$ by $S_x$ (this last random variable is a deterministic function of number theoretic random variables in the coupling, see \eqref{eq:def of Sx and Tx}). More precisely, it is a bound of the form 
        \begin{equation}
        \label{eq:ntbl def}
            \min\bgg\{\pahk\bg[\primes(N_x, x) \in \partial_{S_x} A\bg],\ \pahk\bg[\mathbf V \in \partial_{S_x} A\bg]\bgg\} \ll (\text{NTBL for the set $A$})
        \end{equation}
        for all $x \ge 2$.

        \item A \emph{Poisson--Dirichlet Boundary Lemma} (PDBL), which is an upper bound on the probability that either $\primes(N_x, x)$ or $\mathbf V$ is close to the boundary of $A$ by $T_x$ (this last random variable is a deterministic function of the Poisson--Dirichlet process in the coupling, see \eqref{eq:def of Sx and Tx} again). More precisely, it is a bound of the form 
        \begin{equation}
        \label{eq:pdbl def}
            \min\bgg\{\pahk\bg[\primes(N_x, x) \in \partial_{T_x} A\bg],\ \pahk\bg[\mathbf V \in \partial_{T_x} A\bg]\bgg\} \ll (\text{PDBL for the set $A$})
        \end{equation}
        for all $x\ge 2$.
    \end{itemize}
    \medskip
    
    If we do have an NTBL and a PDBL for some set $A$, then we directly deduce
    \begin{align}
    \begin{split}
    \label{eq:ntbl and pdbl combined}
        \pahk\bgg[\bg\{\primes(N_x, x),\ \mathbf V\bg\}\subseteq \partial_{\max\{S_x, T_x\}} A\bgg] \ll\ &(\text{NTBL for the set $A$})\ + \\
        &(\text{PDBL for the set $A$}).
        \end{split}
    \end{align}
    We note that if the bounds in \eqref{eq:ntbl def} and \eqref{eq:pdbl def} are sharp, then we should not expect a loss in the upper bound \eqref{eq:ntbl and pdbl combined} because we expect two random variables close to each other to be either both close to the boundary or both far from the boundary.

    The minima in \eqref{eq:ntbl def} and \eqref{eq:pdbl def} offer a choice of which quantity to bound to whoever tries to apply this strategy. In practice, we have one of three options depending on whether we have only good number-theoretic information, only good Poisson--Dirichlet information, or both, about proximity to the boundary: Consider the interval $I \coloneqq [\frac{1}{\log x}, \frac{100\log_2 x}{\log x}]$.
    \begin{enumerate}[label=\arabic*), topsep=5pt,itemsep=5pt,partopsep=3pt, parsep=3pt]
        \item Suppose that we know how to compute both the bounds $\nu_x(\partial_\eta A)$ and $\mu(\partial_\eta A)$ for all $\eta \in I$. This is the ideal situation. We would then want to attempt to bound the probabilities $\pahk\bg[\primes(N_x, x)\in \partial_{S_x} A\bg]$ and $\pahk\bg[\mathbf V\in \partial_{T_x} A\bg]$. This is because the event $\{\primes(N_x, x)\in \partial_{S_x} A\}$ belongs to the $\sigma$-algebra generated by the random variable $N_x$, and $\{\mathbf V\in \partial_{T_x} A\}$ belongs to the $\sigma$-algebra generated by $\mathbf V$. Therefore, this approach has the benefit of completely avoiding going through the coupling to compute the desired probabilities.

        \item Suppose now that we have tools to bound $\nu_x(\partial_\eta A)$ for all $\eta \in I$, but we do not know how to compute bounds on $\mu(\partial_\eta A)$. Then we should aim at computing bounds on the probabilities $\pahk\bg[\primes(N_x, x)\in \partial_{S_x} A\bg]$ and $\pahk\bg[\primes(N_x, x)\in \partial_{T_x} A\bg]$. Note that we have no choice here to understand $\pahk\bg[\primes(N_x, x)\in \partial_{T_x} A\bg]$ within the structure of the coupling.

        \item Similarly, if we know how to bound $\mu(\partial_{\eta} A)$, but we have poor understanding of $\nu_x(\partial_\eta A)$ for all $\eta \in I$, then we should aim for $\pahk\bg[\mathbf V\in \partial_{S_x} A\bg]$ and $\pahk\bg[\mathbf V\in \partial_{T_x} A\bg]$. Again, the probability $\pahk\bg[\mathbf V\in \partial_{S_x} A\bg]$ will have to be studied within the coupling.
    \end{enumerate}

    \begin{remss}
        Here are a few remarks about Theorems \ref{thm:ineq under hyp} and \ref{thm:ineq unconditional} and the strategy presented above:
        \begin{enumerate}[label=\alph*), topsep=5pt,itemsep=5pt,partopsep=3pt, parsep=3pt]
            \item If it is not immediately clear what the boundary $\partial A$ is, we have available the identity 
            \begin{equation}
            \label{eq:identity dist to boundary}
                \bg\|\mathbf z - \partial A\bg\| = \max\bgg\{\bg\|\mathbf z-A\bg\|,\ \bg\|\mathbf z-\Delta\setminus A\bg\|\bgg\} \qquad \text{for all $\mathbf z \in \Delta$ and $A \subseteq \Delta$}
            \end{equation}
            to compute the measure of $\partial_\eta A$. This identity is true simply from the fact that $\Delta$ is a convex set. The proof is elementary and we leave it to the reader.

            \item For any integer $n \ge 2$, we could also study the sequence 
            \[
                \primes^*(n) \coloneqq \left(\frac{\log p_i(n)}{\log n}\right),
            \]
            instead of $\primes(n, x)$. Now this sequence is in the set $\Delta^*$ defined as all the sequences $\mathbf x \in \Delta$ satisfying $\|\mathbf x\|=1$. If we define the measure
            \[
                \nu_x^*(A) \coloneqq \frac{1}{\floor x-1}\cdot \#\bgg\{2 \le n \le x : \primes^*(n) \in A\bgg\},
            \]
            then we can obtain analogous versions of Theorems \ref{thm:ineq under hyp} and \ref{thm:ineq unconditional} simply from the fact that
            \[
                \bg\|\primes(n, x) - \primes^*(n)\bg\| = \frac{\log(x/n)}{\log x}
            \]
            for all $n \ge 2$.
            
            \item In \cite[Theorem 2]{hk2025}, Koukoulopoulos and the author also studied the convergence in law of random $k$-factorizations of a random integer to a Dirichlet law using the coupling $(\Omega_\star, \mathcal F_\star, \pahk)$. It is important to note that we cannot write this problem as $\nu_x(A)$ for some set $A$, and thus, we cannot apply directly Theorems \ref{thm:ineq under hyp} or \ref{thm:ineq unconditional}. 
            
            However, a very similar strategy to the one explained earlier was used for this problem: First, we extended the coupling $\Omega_\star$ to include the randomization coming from the $k$-factorization. Then, we introduced two random vectors in $\RR^k$ defined inside this extended coupling: the size of the random $k$-factorization $\bolddelta_{f, x}$, and a Dirichlet-distributed vector $\mathbf Z$. Our Lemma~9.1 in \cite{hk2025} proves that these vectors are close in $\RR^k$ to each other with high probability, and Lemma~9.2 reduces the problem to something resembling an NTBL and a PDBL lemma. Finally, since we had good number-theoretic and Poisson--Dirichlet methods to study the boundary events, we did not need to dive into the coupling at all to prove these two boundary lemmas.
        \end{enumerate}
    \end{remss}

    \subsection{Reduction of Theorems \ref{theo:asymp for hxyz} and \ref{thm:explicit formula for h(u,v)} using Theorem \ref{thm:ineq under hyp}}

    We reduce Theorem \ref{theo:asymp for hxyz} and Theorem \ref{thm:explicit formula for h(u,v)} to a few key Main Lemmas using Theorem \ref{thm:ineq under hyp}. These Main Lemmas will be proven in later sections.

    Recall that we can factorize uniquely any positive integer $n$ into $n=n^\flat n^\sharp$ with $n^\flat$ being squarefree, $n^\sharp$ being squarefull and $\gcd(n^\flat, n^\sharp) = 1$. We define the function
        \begin{equation}
        \label{eq:def of G(x,y)}
            G(x, y) \coloneqq \#\bggg\{n \le x : \exists \, d|n \text{ satisfying } d \in \bgg(y(x/n^\flat)^{-5},\ y(x/n^\flat)^{5}\bgg)\bggg\}.
        \end{equation}
    The first Main Lemma shows that we only need to focus on the case $z \le \sqrt x$ in Theorem \ref{theo:asymp for hxyz} using the symmetry of the set of divisors of $n$.

    \begin{mainlemma}
    \label{lem:reduction to z^2<x}
        We have
        \[
            H(x, y, z) = x \cdot \nu_x\bggg(D\bgg(\tfrac{\log \ymin}{\log x}, \tfrac{\log \zmin}{\log x}\bgg)\bggg) + O\bgg(G(x, \ymin) + G(x, \zmin) + \sqrt x\bgg)
        \]
        for all $3 \le y < z \le x/3$ with $\ymin \coloneqq \min\{y, x/z\}$ and $\zmin \coloneqq \min\{z, x/y, \sqrt x\}$.
    \end{mainlemma}
    Using Theorem \ref{thm:ineq under hyp}, we have a bridge between the measure $\nu_x$ and $\mu$. For all $0 \le u < v \le 1$, we have
    \begin{align}
    \begin{split}
    \label{eq:bridge for D(u,v)}
        \nu_x\bg(D(u, v)\bg) = \mu\bg(D(u, v)\bg) + &O\bggg(\pahk\bg[\primes(N_x, x) \in \partial_{S_x} D(u, v)\bg] \\
        &+ \pahk\bg[\primes(N_x, x) \in \partial_{T_x} D(u, v)\bg] + \frac{1}{\log x}\bggg).
        \end{split}
    \end{align}
    We simplify the events described in the probabilities in the error term. If $\mathbf z \in \partial_\eta D(u, v)$ for some $\eta \ge 0$, then both $\|\mathbf z - \mathbf w\|$ and $\|\mathbf z - \mathbf y\|$ are $< \eta$ for some $\mathbf w \in D(u, v)$ and some $\mathbf y \notin D(u, v)$ using \eqref{eq:identity dist to boundary}. Thus, there must exist a subset $E\subseteq \NN$ satisfying $\mathbf w(E) \in (u, v)$ and $\mathbf y(E) \notin (u, v)$, where we recall $\mathbf x(E) \coloneqq \sum_{i \in E} x_i$. It follows from the fact that the map $\mathbf x \mapsto \mathbf x(E)$ is $1$-Lipschitz that $\mathbf z(E) \in (u-\eta, u+\eta) \cup (v-\eta, v+\eta)$. Therefore, 
    \begin{equation}
        \label{eq:simplified boundary of D(u,v)}
        \partial_\eta D(u, v) \subseteq D(u-\eta, u+\eta) \cup D(v-\eta, v+\eta).
    \end{equation}
    With this inequality, we simplify the first probability in the error term of \eqref{eq:bridge for D(u,v)} as
    \begin{equation}
    \label{eq:simplify NTBL}
        \pahk\bgg[\primes(N_x, x) \in \partial_{S_x} D(u,v)\bgg] \ll \frac{G(x, x^{u}) + G(x, x^v)}{x},
    \end{equation}
    and we simplify the second probability as
    \begin{align}
        \begin{split}
        \label{eq:simplify PDBL}
            \pahk\bgg[\primes(N_x, x) \in \partial_{T_x} D(u,v)\bgg] &\ll \pahk\bgg[\exists\ d|N_x \text{ satisfying } d \in (x^{u-T_x},  x^{u+T_x})\bgg] \\& \ + \pahk\bgg[\exists\ d|N_x \text{ satisfying } d \in (x^{v-T_x}, x^{v+T_x})\bgg].
        \end{split}
    \end{align}
    This brings us to the following two boundary lemmas:
    \begin{mainlemma}[NTBL for the set $D(u, v)$]
        \label{lem:nt boundary lemma}
        Let $G(x, y)$ be the function defined in \eqref{eq:def of G(x,y)}. We have 
        \[
            G(x, y) \ll \frac{x}{(\log y)^\delta(\log_2 y)^{3/2}}
        \]
        for all $3 \le y \le \sqrt x$.
    \end{mainlemma}
    
    \begin{mainlemma}[PDBL for the set $D(u, v)$]
        \label{lem:pd boundary lemma}
        Let $(\Omega_\star, \mathcal F_\star, \pahk)$ be the coupling constructed in \cite{hk2025}, and let $T_x$ be defined as in \eqref{eq:def of Sx and Tx}. We have
        \[
            \pahk\bgg[\exists\ d|N_x \text{ satisfying } d \in (yx^{-T_x}, yx^{T_x})\bgg] \ll \frac{1}{(\log y)^\delta (\log_2 y)^{3/2}}
        \]
        for all $3 \le y \le \sqrt x$.
    \end{mainlemma}
    By combining Main Lemmas \ref{lem:reduction to z^2<x}, \ref{lem:nt boundary lemma} and \ref{lem:pd boundary lemma} as well as \eqref{eq:bridge for D(u,v)}, \eqref{eq:simplify NTBL} and \eqref{eq:simplify PDBL}, we obtain
    \[
        H(x, y, z) = x \cdot \mu\bgg(D\bg(\tfrac{\log \ymin}{\log x}, \tfrac{\log \zmin}{\log x}\bg)\bgg) + O\bggg(\frac{x}{(\log \ymin)^\delta(\log_2 \ymin)^{3/2}}\bggg).
    \]
    For all $0 \le u < v \le 1$, we have $\Delta^* \cap D(u, v) = \Delta^* \cap D(\umin, \vmin)$, with $\umin \coloneqq \min\{u, 1-v\}$ and $\vmin\coloneqq \min\{v, 1-u, \frac{1}{2}\}$. Since the measure $\mu$ is supported on $\Delta^*$,  
    \[
        \mu\bg(D(u, v)\bg) = \mu\bg(D(\umin, \vmin)\bg).
    \]
    By definition of the function $h(u, v)$ in \eqref{eq:tenenbaum limit result}, 
    \begin{equation}
    \label{eq:h(u,v) equals mu}
        h(u, v) = \mu\bg(D(u, v)\bg).
    \end{equation}
    This proves Theorem \ref{theo:asymp for hxyz}. Finally, we directly reduce Theorem \ref{thm:explicit formula for h(u,v)} to the following lemma using \eqref{eq:density of PD} and \eqref{eq:h(u,v) equals mu}:
    \begin{mainlemma}
    \label{lem:polytopes}
        Let $0 \le u < v \le 1$, let $d = \floor{\frac{1}{v-u}}$. Whether a sequence $\mathbf x \in \Delta^*$ has a subsum in $(u, v)$ is completely determined by its first $d$ terms. More precisely, $\mathbf x \notin D(u, v)$ if and only if, for every $E \subseteq \{1, \ldots, d\}$, one of the inequalities 
        \[
            \sum_{\substack{1 \le i \le d \\ i \in E}} x_i \ge v \quad \text{or} \quad \sum_{\substack{1 \le i \le d \\ i \notin E}} x_i \ge 1-u
        \]
        holds. The choice between the two inequalities may depend on $E$.
    \end{mainlemma}

    All that is left is to prove Main Lemmas \ref{lem:reduction to z^2<x}, \ref{lem:nt boundary lemma}, \ref{lem:pd boundary lemma} and \ref{lem:polytopes}. The proof of Main Lemma \ref{lem:pd boundary lemma} will be the longest part since we will have to dive into the structure of the coupling.

    \begin{remss}
        (a) As discussed earlier, we could apply Theorem \ref{thm:ineq unconditional} instead of Theorem \ref{thm:ineq under hyp} to completely avoid the boundary lemmas at the cost of arriving at a weaker version of Theorem \ref{theo:asymp for hxyz}. We use \eqref{eq:ford result}, \eqref{eq:H(x,y,z) in prob language}, \eqref{eq:simplified boundary of D(u,v)}, \eqref{eq:h(u,v) equals mu} and Theorem \ref{thm:ineq unconditional} with $\epsilon = 1$ and $\eta = \frac{5}{\log x} (\delta \log_2 y + (\frac{3}{2} - \delta)\log_3 y)$ to directly obtain
        \begin{equation}
        \label{eq:weaker H(x,y,z) estimate}
            H(x, y, z) = x \cdot h(\tfrac{\log y}{\log x}, \tfrac{\log z}{\log x}) + O\bggg(\frac{x}{(\log y)^\delta(\log_2 y)^{3/2-\delta}}\bggg)
        \end{equation}
        for all $3 \le y < z \le \sqrt x$.
        \medskip

        (b) In \cite[Corollary 1.1]{hk2025}, we showed that if $d^*(n)$ is the function $\min\{d|n : d \ge \sqrt n\}$, then there exists some $c \in (\frac{1}{2}, 1)$ such that
        \[
            \sum_{n \le x} \log d^*(n) = c \cdot x\log x + O(x).
        \]
        This constant $c$ was equal to $\EE_\mu[\Psi]$, with $\EE_\mu$ being the expectation under the measure $\mu$ and $\Psi$ being the random variable on $\Delta^*$ defined by 
        \[
            \Psi(\mathbf x) \coloneqq \inf \bg(\{\mathbf x(E) : E \subseteq \NN\} \cap [\tfrac{1}{2}, 1]\bg).
        \]
        Since $\Psi$ is positive, we have
        \[
            c = \EE_\mu[\Psi] = \int_0^\infty \mu(\Psi \ge u) \dee u. 
        \]
        By definition, we always have $\Psi \in [\frac{1}{2}, 1]$, and the event $\{\Psi \ge u\}$ is equal almost surely to the event $D(1-u, \frac{1}{2})^c$. With \eqref{eq:h(u,v) equals mu}, we have
        \[
            c = 1 - \int_0^{\frac{1}{2}} h(u, \tfrac{1}{2}) \dee u,
        \]
        which can be estimated with any desired precision using Theorem \ref{thm:explicit formula for h(u,v)}.
    \end{remss}

\section{Short proofs of Main Lemmas \ref{lem:reduction to z^2<x}, \ref{lem:nt boundary lemma} \& \ref{lem:polytopes} }
\label{sec:short proofs}

The only thing in common between the proofs of Main Lemmas \ref{lem:reduction to z^2<x}, \ref{lem:nt boundary lemma} and \ref{lem:polytopes} is the fact that they can be done quickly. That is why we collect them in this section.

\begin{proof}[Proof of Main Lemma \ref{lem:reduction to z^2<x}]
    The function $H(x, y, z)$ counts the number of integers $n \le x$ which have a divisor $d|n$ with $d \in (y, z)$. By the symmetry of the set of divisors, any $n$ has a divisor in $(y, z)$ if and only if it also has one in the open interval
    \[
        \bgg(\min\{y, n/z\}, \min\{z, n/y, \sqrt n\}\bgg),
    \]
    unless $n$ is a perfect square. Thus, we have 
    \[
        |H(x, y, z) - H(x, \ymin, \zmin)| \le S_1+S_2+S_3+1,
    \]
    with $S_1$ being the number of integers $n < x$ with a divisor $d|n$ satisfying
    \[
        d \in \bgg(\min\{y, n/z\}, \min\{y, x/z\}\bgg] \subseteq \bgg(\ymin \cdot (x/n^\flat)^{-5}, \ymin \cdot (x/n^\flat)^{5}\bgg),
    \]
    with $S_2$ being the number of integers $n < x$ with a divisor $d|n$ satisfying
    \[
        d \in \bgg[\min\{z, n/y,\sqrt n\}, \min\{z, x/y, \sqrt x\}\bgg) \subseteq \bgg(\zmin \cdot (x/n^\flat)^{-5}, \zmin \cdot (x/n^\flat)^{5}\bgg),
    \]
    and $S_3$ being the number of perfect squares that are $\le x$. We have $S_1 \le G(x, \ymin)$, $S_2 \le G(x, \zmin)$, and $S_3 \le \sqrt x$. We conclude the proof by using the fact that 
    \[
        H(x, \ymin, \zmin) = x \cdot \nu_x\bgg(D\bg(\tfrac{\log \ymin}{\log x}, \tfrac{\log \zmin}{\log x}\bg)\bgg) + O(1),
    \]
    as we saw in \eqref{eq:H(x,y,z) in prob language}.
\end{proof}

\begin{proof}[Proof of Main Lemma \ref{lem:nt boundary lemma}]
Let $3 \le y \le \sqrt x$ and positive integers $b, j$, with $b$ being squarefull. We consider the sets 
\begin{align*}
    \mathcal G(x, y) &\coloneqq \{n \in \NN \cap [1, x] : \exists\ d|n \text{ satisfying } d \in (y(x/n^\flat)^{-5}, y(x/n^\flat)^5)\}, \\
    \mathcal K(b, j, x) &\coloneqq \{n \in \NN \cap [1, x] : n^\sharp = b \text{ and } n^\flat \in (\tfrac{x}{2^{j}b}, \tfrac{x}{2^{j-1}b}]\}.
\end{align*}
The cardinality of the set $\mathcal G(x, y)$ is $G(x, y)$. Since $\# \mathcal K(b, j, x) \le \frac{x}{2^{j-1}b}$, we have 
\[
    \underset{\substack{b, j \ge 1 :\ b \text{ squarefull} \\ 2^jb > \log y}}{\sum\sum} \#\mathcal K(b, j, x)  \le \frac{2x}{(\log y)^{1/3}} \sum_{\substack{b \ge 1 \\ b \text{ squarefull}}} \frac{1}{b^{2/3}} \sum_{j \ge 1} \frac{1}{2^{2j/3}}.
\]
The two series on the right-hand side are convergent, hence
\begin{equation}
    \label{eq:G(y,x) decomposition}
    G(x, y) \ll \underset{\substack{b, j \ge 1 :\ b \text{ squarefull} \\ 2^jb \le \log y}}{\sum\sum} \#\bg(\mathcal G(x, y) \cap \mathcal K(b, j, x)\bg) + \frac{x}{(\log y)^{1/3}}.
\end{equation}
Consider the map $\psi \colon \mathcal G(x, y) \cap \mathcal K(b, j, x) \to \NN$ defined by $\psi(n) = n^\flat$. Since every element of the domain has the same squarefull part, this function is injective. Furthermore, for every integer $n \in \mathcal G(x, y) \cap \mathcal K(b, j, x)$, there is a divisor $d|n$ in the interval $d \in (y(2^jb)^{-5}, y(2^jb)^5)$, so we have
\[
    \gcd(d, \psi(n)) = \frac{d}{\gcd(d, b)} \in (y(2^jb)^{-6}, y(2^jb)^5).
\]
Therefore, any $m$ in the image $\psi\bg(\mathcal G(x, y) \cap \mathcal K(b, j, x)\bg)$ must satisfy $m \in \NN \cap [1, \frac{x}{2^{j-1}b}]$ and must have a divisor in $(y(2^jb)^{-6}, y(2^jb)^5)$. Since $\psi$ is injective, we must have
\begin{align*}
    \#\bg(\mathcal G(x, y) \cap \mathcal K(b, j, x)\bg) &= \#\psi\bg(\mathcal G(x, y) \cap \mathcal K(b, j, x)\bg) \\
    &\le H\bgg(\tfrac{x}{2^{j-1}b},\ y(2^jb)^{-6},\ y(2^jb)^{5}\bgg) \\
    &\ll \frac{(j+\log b)^\delta}{2^jb} \cdot \frac{x}{(\log y)^\delta(\log_2 y)^{3/2}}
\end{align*}
whenever $2^jb \le \log y$ by using \eqref{eq:ford result} in the last inequality. We insert this bound in \eqref{eq:G(y,x) decomposition}, and Main Lemma~\ref{lem:nt boundary lemma} immediately follows.
\end{proof}

\begin{proof}[Proof of Main Lemma \ref{lem:polytopes}]
    We first prove the reverse implication. Suppose that $\mathbf x \in \Delta^*$ and that, for every $F \subseteq [d] \coloneqq \{1, \ldots, d\}$, we either have $\mathbf x(F) \ge v$ or $\mathbf x([d]\setminus F) \ge 1-u$. Then given $E \subseteq \NN$, we have $\mathbf x(E) \ge \mathbf x(E \cap [d]) \ge v$ or $\mathbf x(E) = 1-\mathbf x(E^c) \le 1-\mathbf x(E^c \cap [d]) \le u$. Therefore, $\mathbf x$ cannot have a subsum in the interval $(u, v)$.
    
    Conversely, suppose that $\mathbf x \in \Delta^* \cap D(u, v)^c$. Consider first $E$ such that $\mathbf x(E) \ge v$. We construct a sequence $(\gamma_j)_{j \ge 0}$ by defining $\gamma_0 \coloneqq 0$ and $\gamma_j \coloneqq \mathbf x(E \cap [j])$ for all $j \ge 1$. Since the sequence is non-decreasing and $\lim_{j \to \infty} \gamma_j \ge v$, there is a first index $\ell$ satisfying $\gamma_{\ell} \ge v$. Since no $\gamma_j$ lies in the interval $(u, v)$, we have $\gamma_{\ell-1} \le u$ and hence $x_\ell = \gamma_\ell - \gamma_{\ell-1} \ge v-u$. Therefore, 
    \[
        \ell \le \#\{i \ge 1 : x_i \ge v-u\} \le \sum_{i=1}^\infty \frac{x_i}{v-u} = \frac{1}{v-u}.
    \]
    Thus, since $d = \floor{\frac{1}{v-u}}$, we have $\ell \le d$, and therefore $\mathbf x(E \cap [d]) = \gamma_d \ge \gamma_\ell \ge v$. If instead $\mathbf x(E) \le u$, then $\mathbf x(E^c) = 1-\mathbf x(E) \ge 1-u$. Since $\mathbf x$ does not have a subsum in $(1-v,1-u)$, a similar argument shows that $\mathbf x(E^c \cap [d]) \ge 1-u$.
\end{proof}

\section{The coupling}
\label{sec:coupling}

The proof of Main Lemma \ref{lem:pd boundary lemma} requires us to understand further the structure of the coupling in \cite{hk2025}. 

Here is one way to construct this coupling. Let $(\Omega_\star, \mathcal F_\star, \pahk)$ be any ambient probability space containing the following random objects:
\begin{itemize}
    \item A Poisson point process $\RS$ in $\RR_{>0} \times \RR_{>0}$ with intensity $e^{-wy} \dee w \dee y$.\footnote{The coupling in \cite{hk2025} was actually defined as any ambient space with a GEM process and three i.i.d. uniform random variables that are also independent of the GEM. As we will see later in Proposition \ref{prop:dist of V}, the GEM process can be extracted from the Poisson process $\RS$.}

    \item Three i.i.d. random variables $U_1, U_2, U_3$ that are uniformly distributed in $(0, 1)$, and which are also independent of $\RS$.
\end{itemize}
From these random objects, we explain how to extract a random integer $N_x$ and a Poisson--Dirichlet process $\mathbf V$. All the following random variables were introduced in \cite{arratia02} or in \cite{hk2025}, except $M_x^*$, $\thetasmall$ and $\thetalarge$. We will write $\eahk$ as the expectation inside this coupling.

\begin{deff}[$x$-labelling of $\RS$]
    \label{deff:x-label of R}
    Almost surely, we may find a unique sequence of random points~$(W_i, Y_i)_{i \in \ZZ}$ satisfying
    \begin{itemize}
        \item $\RS = \bg\{(W_i, Y_i) : i \in \ZZ\bg\}$;
        \item $W_i < W_{i+1}$ for all $i\in\ZZ$;
        \item if we let $S_i\coloneqq \sum_{\ell \ge i} Y_\ell$ for all $i\in\ZZ$, then we have $S_1 \le \log x < S_0$.
    \end{itemize}
    We refer to this sequence as the \emph{$x$-labelling} of the points of $\RS$.
\end{deff}

\begin{deff}[The processes $\mathbf L$ and $\mathbf V$]
    \label{deff:process}
    Let $x \ge 2$, let $(W_i, Y_i)_{i \in \ZZ}$ be the $x$-labelling of $\RS$, and let $S_i \coloneqq \sum_{\ell \ge i} Y_\ell$ for all $i \in \ZZ$. We define the process $\mathbf L = (L_1, L_2, \ldots)$ by $L_1 \coloneqq 1-\frac{S_1}{\log x}$, and $L_i \coloneqq \frac{Y_{i-1}}{\log x}$ for all $i \ge 2$. We also define the process $\mathbf V$ as the rearrangement of the terms of the process $\mathbf L$ in non-increasing order.
\end{deff}

\begin{deff}[The deterministic functions $h$ and $r$]
    \label{deff:functions h and r}
    Let $(\lambda_j)_{j \ge 0}$ be the increasing sequence of positive real numbers defined by $\lambda_0 \coloneqq e^{-\gamma}$ and 
    \[
        \lambda_j \coloneqq \exp\bggg(-\gamma + \sum_{1 \le i \le j} \frac{1}{v_iq_i}\bggg) \text{ for $j \ge 1$,}
    \]
    with $\gamma$ being the Euler-Mascheroni constant and $q_i = p_i^{v_i}$ being the $i$\textsuperscript{th} smallest prime power, i.e., $(q_i)_{i \ge 1}$ is the sequence $2, 3, 2^2, 5, 7, 2^3, 3^2, \ldots$ We define the step-function $h \colon \RR_{>0} \to \RR_{\ge 0}$ by
    \[
        h(t) \coloneqq \sum_{j \ge 1} (\log q_j)\one_{\lambda_{j-1} < t \le \lambda_{j}},
    \]
    and we define the function $r \colon \RR_{>0} \to \RR_{\ge 0}$ as 
    \[
        r(t) \coloneqq |h(t) - t|.
    \]
\end{deff}

\begin{deff}[The Poisson process $\RS^*$ and its $x$-labelling]
    \label{deff:PP RS* and label}
    Let $h$ be the deterministic function from Definition \ref{deff:functions h and r}. Consider the map $\psi \colon \RR_{>0} \times \RR_{>e^{-\gamma}} \to \RR_{>0} \times \mathcal Q$, with $\mathcal Q$ being the set of prime powers, and 
    \[
        \psi(w, y) \coloneqq \bgg(\frac{wy}{h(y)}, e^{h(y)}\bgg).
    \]
    Define the Poisson process $\RS^* \coloneqq \bg\{\psi(W, Y) : (W, Y) \in \RS \text{ and } Y > e^{-\gamma}\bg\}$. The \emph{$x$-labelling} of $\RS^*$ is the almost surely unique sequence of random points~$(T^*_i, Q_i^*)_{i \in \ZZ_{\le K}}$ along with the random non-negative integer $K$ satisfying
		\begin{itemize}
			\item $\RS^* = \bg\{(T_i^*, Q_i^*) : i \in \ZZ_{\le K}\bg\}$;
			\item $T^*_i < T^*_{i+1}$ for all $i\in\ZZ_{\le K}$;
			\item $\prod_{i=1}^K  Q_i^*\le x< \prod_{i=0}^K Q_i^*$.
		\end{itemize}
\end{deff}

\begin{deff}[The random integer $M_x$]
    \label{deff:def of M}
    Let $(W_i, Y_i)_{i \in \ZZ}$ be the $x$-labelling of $\RS$ (see Definition~\ref{deff:x-label of R}). We define $J_x \coloneqq \prod_{i \ge 1} e^{h(Y_i)}$. We also define the \emph{extra prime} $\pextra$ to be the smallest element of $\{1\} \cup \{\text{primes}\}$ such that $\theta(\pextra) \ge U_1\theta(x/J_x)$, where $\theta(y) = \sum_{p \le y} \log p$ is Chebyshev's function. Finally, we define the random integer $M_x \coloneqq J_x\pextra$.
\end{deff}

\begin{deff}[The random integer $M_x^*$]
    \label{deff:def of M*}
    Let $(T_i^*, Q_i^*)_{i \in \ZZ_{\le K}}$ be the $x$-labelling of $\RS^*$ (see Definition~\ref{deff:PP RS* and label}). We define $J^*_x \coloneqq \prod_{1 \le i \le K} Q^*_i$. We also define the \emph{extra prime} $\pextra^*$ to be the smallest element of $\{1\} \cup \{\text{primes}\}$ such that $\theta(\pextra^*) \ge U_1\theta(x/J_x^*)$, where $\theta(y) = \sum_{p \le y} \log p$ is Chebyshev's function. Finally, we define the random integer $M^*_x \coloneqq J_x^*\pextra^*$.
\end{deff}

\begin{deff}[The random integer $N_x$]
    \label{deff:def of N}
    For any two probability measures $\mu_1$ and $\mu_2$ supported on $\NN$, we define a sequence $(z_j)_{j \ge 0}$ with $z_0 \coloneqq 0$ and 
    \[
        z_j \coloneqq \sum_{1 \le i \le j} \frac{(\mu_1(i) - \mu_2(i))^-}{\dtv(\mu_1, \mu_2)},
    \]
    with $\dtv(\mu_1, \mu_2) \coloneqq \sup_{A \subseteq \NN} |\mu_1(A) - \mu_2(A)|$ being the total variation distance between $\mu_1$ and $\mu_2$. Let $f_{\mu_1, \mu_2} \colon \NN \times (0, 1)^2 \to \NN$ be defined as 
    \[
        f_{\mu_1, \mu_2}(m, a, b) = \begin{cases}
            m &\text{if $a \cdot \mu_1(m) \le \mu_2(m)$,} \\
            \sum_{i \ge 1} i \cdot \one_{z_{i-1} < b \le z_i} &\text{otherwise.}
        \end{cases}
    \]
    We define the random integer 
    \[
        N_x \coloneqq f_{\mu_1, \mu_2}(M_x, U_2, U_3)
    \]
    with $\mu_1$ being the law of $M_x$ and $\mu_2$ being the uniform distribution on $\NN \cap[1, x]$.
\end{deff}

\begin{deff}[The random variables $\thetasmall$ and $\thetalarge$]
    \label{deff:def of Thetas}
    Let $r$ be the function from Definition \ref{deff:functions h and r}. We define $r_0 \coloneqq \sup_{t > 0} r(t)$ and 
    \[
        \thetasmall \coloneqq r_0 + \sum_{\substack{(W, Y) \in \RS \\ Y \le e^{-\gamma}}} r(Y).
    \]
    Let $(W_i, Y_i)_{i \in \ZZ}$ be the $x$-labelling of $\RS$ (see Definition \ref{deff:x-label of R}). We also define $\thetalarge$ as the random variable 
    \[
        \thetalarge \coloneqq \sum_{\substack{i \ge 1 \\ Y_i > e^{-\gamma}}} r(Y_i).
    \]
\end{deff}

\bigskip
Here are some properties that these random objects satisfy:

\begin{prop}[Distribution of $N_x$]
    The random integer $N_x$ is uniformly distributed in $\NN \cap [1, x]$.
\end{prop}

\begin{proof}
    In general, if $X$ is any random variable with law $\mu_1$ and $U, U'$ are two uniform random variables independent of each other and of $X$, then it was shown in \cite[Section 3.8]{arratia02} that the law of the random variable $f_{\mu_1, \mu_2}(X, U, U')$ is $\mu_2$, with $f_{\mu_1, \mu_2}$ being defined as in Definition \ref{deff:def of N}. This was also reproven in \cite[Lemma B.2]{hk2025} using the notation of Definition \ref{deff:def of N}.
\end{proof}

\begin{prop}[Distribution of $\mathbf V$]
\label{prop:dist of V}
    The process $\mathbf V$ follows a Poisson--Dirichlet distribution.
\end{prop}

\begin{proof}
    The process $\mathbf L$ from Definition~\ref{deff:process} follows a GEM distribution (see \cite[Proposition~4.3]{hk2025}). If we rearrange the terms of any GEM process in non-increasing order, then we always obtain a Poisson--Dirichlet process (see \cite[Theorem~2.7]{feng2010}).
\end{proof}

\begin{prop}[Total variation distance between $N_x, M_x$ and $M_x^*$]
    \label{prop:dtv of N M M*}
    For all $x \ge 2$, let $\mathcal E(x) \coloneqq \{N_x = M_x = M_x^*\}$. We have 
    \[
        \pahk\bg[\mathcal E(x)^c\bg] \ll \frac{1}{\log x}.
    \]
\end{prop}

\begin{proof}
    Note that if $N_x = M_x$ and $J_x = J_x^*$ (with $J_x$ and $J_x^*$ being defined as in Definitions \ref{deff:def of M} and \ref{deff:def of M*}), then $\mathcal E(x)$ must occur. Therefore, 
    \[
        \pahk\bg[\mathcal E(x)^c\bg] \le \pahk\bg[N_x \ne M_x\bg] + \pahk\bg[J_x \ne J_x^*\bg].
    \]
    It follows from \cite[Propositions 2.4 and 6.5]{hk2025} that the right-hand side of the inequality above is $O(\frac{1}{\log x})$.
\end{proof}

\begin{prop}[Bound on $r(t)$]
    \label{prop:bound on r(t)}
    For all $t > 0$, we have the upper bound $r(t) \ll \min\{t, t^{-2}\}$.
\end{prop}

\begin{proof}
    This follows from the Prime Number Theorem. See \cite[Section 2]{hk2025} for more details.
\end{proof}

\begin{prop}[Moment generating function of $\thetasmall + \thetalarge$]
    \label{prop:mgf of Theta}
    Fix $\alpha > 0$, and let $\thetasmall$ and $\thetalarge$ be defined as in Definition \ref{deff:def of Thetas}. We have 
    \[
        \eahk\bg[e^{\alpha(\thetasmall + \thetalarge)}\bg] \ll_\alpha 1.
    \]
    In particular, we have 
    \[
    \pahk\bg[\thetasmall+\thetalarge > \alpha\cdot\log_2 y\bg] \ll_\alpha \frac{1}{\log y}
    \]
    for any $y \ge 2$.
\end{prop}

\begin{proof}
    The proof of this proposition is very similar to the proof of \cite[Lemma 2.3]{hk2025}: Let $\Theta_\infty$ be defined as 
    \[
        \Theta_\infty \coloneqq r_0 + \sum_{(W, Y) \in \RS} r(Y).
    \]
    We have $\thetasmall+\thetalarge \le \Theta_\infty$. By using Campbell's Theorem (see \cite[Section 3.2]{kingman1993})
    \[
        \eahk\bg[e^{\alpha(\thetasmall + \thetalarge)}\bg] \le \eahk\bg[e^{\alpha\Theta_\infty}\bg] = \exp\bggg(\alpha r_0 + \int_0^\infty \frac{e^{\alpha r(y)}-1}{y} \dee y\bggg) \ll_\alpha 1,
    \]
    with the convergence of the integral for all fixed $\alpha > 0$ following from $e^{\alpha r(y)} - 1 \ll_\alpha r(y) \ll \min\{y, y^{-2}\}$ (see Proposition \ref{prop:bound on r(t)}). The bound on the probability $\pahk[\thetasmall + \thetalarge > \log_2 y]$ is then a direct application of Chernoff's inequality.
\end{proof}

\section{Proof of Main Lemma \ref{lem:pd boundary lemma}}
\label{sec:pdbl}

Recall the construction of the coupling $(\Omega_\star, \mathcal F_\star, \pahk)$ from the previous section, as well as all the random variables that were introduced there. Recall also the definition of $T_x$ in \eqref{eq:def of Sx and Tx}. We define the infinite set $\mathcal H(y, \kappa)$ and the function $\xi(y)$ as
\begin{align*}
    \mathcal H(y, \kappa) &\coloneqq \bgg\{n \in \NN :  \exists\ d|n \text{ such that } d \in (y/\kappa, y\cdot \kappa)\bgg\} \\
    \xi(y) &\coloneqq (\log y)^{-\delta}(\log_2 y)^{-3/2}.
\end{align*}
We have to keep in mind the monotonicity of the set $\mathcal H(y, \kappa)$ with respect to $\kappa$, i.e., if $1 < \kappa_1 \le \kappa_2$, then $\mathcal H(y, \kappa_1) \subseteq \mathcal H(y, \kappa_2)$. We can rewrite what we need to show in Main Lemma \ref{lem:pd boundary lemma} as
\begin{equation}
    \label{eq:rewriting PDBL}
    \pahk\bgg[N_x \in \mathcal H(y, x^{T_x})\bgg] \ll \xi(y).
\end{equation}
To prove Main Lemma \ref{lem:pd boundary lemma}, we will need four lemmas:
\begin{lemma}
    \label{lem:H set intersects N M M*}
    If $x, y, \kappa$ are real numbers satisfying $2 \le \kappa \le (\log y)^{40}$ and $3 \le y \le x^{2/3}$, then 
    \[
        \pahk\bgg[\{N_x, M_x, M_x^*\} \text{ \emph{intersects} } \mathcal H(y, \kappa)\bgg] \ll (\log \kappa)^\delta \cdot \xi(y).
    \]
\end{lemma}

\begin{proof}
    Define the event $\mathcal E(x) \coloneqq \{N_x = M_x = M_x^*\}$. If $\{N_x, M_x, M_x^*\}$ intersects $\mathcal H(y, \kappa)$, then either $N_x \in \mathcal H(y, \kappa)$ or $\mathcal E(x)$ does not occur. Therefore, by Proposition \ref{prop:dtv of N M M*}, we have 
    \[
        \pahk\bgg[\{N_x, M_x, M_x^*\} \text{ intersects } \mathcal H(y, \kappa)\bgg] \ll \pahk\bgg[ N_x \in \mathcal H(y, \kappa)\bgg] + (\log x)^{-1}.
    \]
    We have 
    \[
        \pahk\bgg[ N_x \in \mathcal H(y, \kappa)\bgg] \ll (\log \kappa)^\delta \cdot \xi(y)
    \]
    by \cite[Theorem 1]{ford2008}.
\end{proof}

\begin{lemma}
\label{lem:multiples in H}
    Let $x, y, \kappa$ be real numbers satisfying $2 \le \kappa \le (\log y)^{30}$ and $3 \le y \le \sqrt x$. Let $d$ be a positive integer satisfying $d \le \sqrt y$. Then
    \[
        \pahk\bgg[d|N_x \text{ and } N_x \in \mathcal H(y, \kappa)\bgg] \ll \frac{\tau(d)}{d} \cdot (\log \kappa)^\delta \cdot \xi(y),
    \]
    with $\tau(d)$ being the number of positive divisors of $d$.
\end{lemma}

\begin{proof}
    Suppose that $a$ and $b$ are positive integers satisfying $a|db$ and that $a \in (y/\kappa, y\cdot \kappa)$. We must have $\frac{a}{\gcd(a, b)}|d$. Therefore, the integer $\gcd(a, b)$ is in the set $\{a/s : s|d\}$, which is itself a subset of the union of open intervals $\bigcup_{s|d} (\frac{y}{s\kappa}, \frac{y\kappa}{s})$. It follows that if $db \in \mathcal H(y, \kappa)$, then $b \in \bigcup_{s|d} \mathcal H(y/s, \kappa)$. Thus, we have 
    \begin{align*}
        \pahk\bgg[d|N_x \text{ {and} } N_x \in \mathcal H(y, \kappa)\bgg] &\le \sum_{s|d} \pahk\bgg[d|N_x \text{ {and} } N_x/d \in \mathcal H(y/s, \kappa)\bgg] \\
        &= \frac{\floor{x/d}}{\floor x} \sum_{s|d} \pahk\bgg[N_{x/d} \in \mathcal H(y/s, \kappa)\bgg] \\
        &\ll \frac{\tau(d)}{d} \cdot \max_{s|d} \pahk\bgg[N_{x/d} \in \mathcal H(y/s, \kappa)\bgg].
    \end{align*}
    By applying Lemma \ref{lem:H set intersects N M M*}, we finish the proof.
\end{proof}

\begin{lemma}
\label{lem:thetasmall H bound}
    For all $3 \le y \le \sqrt x$, we have 
    \[
        \pahk\bgg[M_x^* \in \mathcal H(y, e^{10\cdot \thetasmall})\bgg] \ll \xi(y).
    \]
\end{lemma}

\begin{proof}
    Let $\Sigma_1$ be the $\sigma$-algebra generated by $U_1$ and the restriction of $\RS$ on the set $(0, \infty) \times (e^{-\gamma}, \infty)$, and let $\Sigma_2$ be the $\sigma$-algebra generated by the restriction of $\RS$ on the set $(0, \infty) \times (0, e^{-\gamma}]$. Note that $\Sigma_1$ and $\Sigma_2$ are independent $\sigma$-algebras. Since $M_x^*$ is $\Sigma_1$-measurable and $\thetasmall$ is $\Sigma_2$-measurable, we conclude that $M_x^*$ and $\thetasmall$ are independent. Thus, we can bound the conditional expectation 
    \[
        \eahk\bgg[\one_{M_x^* \in \mathcal H(y, e^{10\cdot \thetasmall})} \cdot \one_{\thetasmall \le \log_2 y}\ \bgg|\ \thetasmall\bgg] \ll (\thetasmall)^\delta \cdot \xi(y) \le e^{\thetasmall} \cdot \xi(y).
    \]
    Therefore, 
    \begin{align*}
        \pahk\bgg[M_x^* \in \mathcal H(y, e^{10 \cdot\thetasmall})\bgg] \ll \eahk\bgg[e^{\thetasmall}\bgg] \cdot \xi(y) + \pahk\bgg[\thetasmall > \log_2 y\bgg].
    \end{align*}
    The lemma follows directly from applying Proposition \ref{prop:mgf of Theta}.
\end{proof}

\begin{lemma}
\label{lem:thetalarge H bound}
    For all $3 \le y \le \sqrt x$, we have 
    \[
        \pahk\bgg[N_x \in \mathcal H(y, e^{10 \cdot \thetalarge})\bgg] \ll \xi(y).
    \]
\end{lemma}
\begin{proof}
    For any prime power $p^v$, we define $A_{p^v} \coloneqq \#\{i \ge 1 : h(Y_i) = \log p^v\}$. We must have
    \begin{align*}
        A_{p^v} = \sum_{1 \le k < \frac{\log y}{2\log p^v}} \one_{k \le A_{p^v} < \frac{\log y}{2\log p^v}} + A_{p^v}\one_{A_{p^v} \ge \frac{\log y}{2\log p^v}} \le \sum_{1 \le k < \frac{\log y}{2\log p^v}} \one_{A_{p^v} \ge k} + \frac{2\log p^v\cdot A_{p^v}^2}{\log y}.
    \end{align*}
    If $A_{p^v} \ge k$, then $p^{vk}|M_x$. Therefore,
    \begin{align*}
        A_{p^v} \le \sum_{1 \le k < \frac{\log y}{2\log p^v}} \one_{p^{vk}|M_x} + \frac{2\log p^v \cdot A_{p^v}^2}{\log y}.
    \end{align*}
    With this inequality and Proposition \ref{prop:bound on r(t)}, we have
    \begin{align}
    \begin{split}
        \label{eq:thetalarge decomp}
        \thetalarge  &\ll \sum_{\substack{i \ge 1 \\ Y_i > e^{-\gamma}}} \frac{1}{h(Y_i)^2} = \sum_{p^v} \frac{A_{p^v}}{(\log p^v)^2} \\
        &\le \underset{\substack{p \ \text{prime and } v, k \ge 1 \\ p^{vk} \le \sqrt y}}{\sum\sum\sum} \frac{\one_{p^{vk}|M_x}}{(\log p)^2} + \frac{2}{\log y}\sum_{p^v} \frac{A_{p^v}^2}{\log p^v}.
    \end{split}
    \end{align}
    Since $A_{p^v} \le \#\{(W, Y) \in \RS : h(Y) = \log p^v\}$, and this last random variable is Poisson distributed with parameter $\frac{1}{vp^v}$, we have $\EE[A_{p^v}^2] \ll \frac{1}{vp^v}$. It follows that
    \begin{equation}
    \label{eq:large Apv}
        \frac{2}{\log y}\sum_{p^v} \frac{\EE[A_{p^v}^2]}{\log p^v} \ll \frac{1}{\log y}.
    \end{equation}
    
    For $1 \le j \le \log_3 y$, define the events
    \[
        \mathcal B_j \coloneqq \bgg\{N_x \in \mathcal H\bg(y, e^{10e^{j+1}}\bg)\bgg\},
        \qquad
        \mathcal C_j \coloneqq \{\thetalarge \in (e^j, e^{j+1}]\},
        \qquad
        \mathcal E(x) \coloneqq \{N_x = M_x = M_x^*\}.
    \]
    Using Proposition \ref{prop:mgf of Theta}, Lemma \ref{lem:H set intersects N M M*} and the inequalities \eqref{eq:thetalarge decomp} and \eqref{eq:large Apv}, we have the decomposition
    \begin{align*}
        \pahk\bgg[\{N_x \in \mathcal H(y, e^{10\thetalarge})\} \cap \mathcal E(x)\bgg]
        &\le \sum_{1 \le j \le \log_3 y}
            \pahk\bg[\mathcal B_j \cap \mathcal C_j \cap \mathcal E(x)\bg]
            + O\bg(\xi(y)\bg) \\
        &\le \sum_{1 \le j \le \log_3 y} e^{-j}
            \eahk\bg[\thetalarge\one_{\mathcal B_j}\one_{\mathcal E(x)}\bg]
            + O\bg(\xi(y)\bg) \\
        &\ll \sum_{1 \le j \le \log_3 y}\frac{1}{e^j}
            \sum_{\substack{p\text{ prime},\ v,k \ge 1 \\ p^{vk} \le \sqrt y}}
            \frac{\pahk\bg[\{p^{vk}\mid N_x\} \cap \mathcal B_j\bg]}{(\log p)^2}
            + O\bg(\xi(y)\bg).
    \end{align*}
    By applying Lemma \ref{lem:multiples in H}, we arrive at
    \[
        \pahk\bgg[\{N_x \in \mathcal H(y, e^{10\thetalarge})\} \cap \mathcal E(x)\bgg] \ll \xi(y) \cdot \sum_{j\ge 1} e^{j(\delta-1)} \cdot  \sum_p \sum_{v \ge 1} \sum_{k \ge 1} \frac{vk+1}{p^{vk}(\log p)^2}.
    \]
    Using the inequalities $vk+1 \le 2vk$ and $p^{vk} \ge p2^{v+k-2}$ and Proposition \ref{prop:dtv of N M M*}, we have
    \[
        \pahk\bgg[N_x \in \mathcal H(y, e^{10\cdot \thetalarge})\bgg] \ll \xi(y) \cdot \left(\sum_{j\ge 1} e^{j(\delta-1)}\right) \cdot  \left(\sum_p \frac{1}{p(\log p)^2}\right) \cdot  \left(\sum_{v \ge 1}  \frac{v}{2^v}\right)^2 + \frac{1}{\log x}.
    \]
    The series on the right-hand side are all convergent, which proves the lemma.
\end{proof}

\begin{proof}[Proof of Main Lemma \ref{lem:pd boundary lemma}]
    Let $r_0 \coloneqq \sup_{t >0} r(t)$. From the definition of $\mathbf L$, $\mathbf V$, $\thetasmall$ and $\thetalarge$, we have 
    \[
        T_x\cdot (\log x) = 5\sum_{i \ge 1} r(V_i\log x) = 5\sum_{i \ge 1}r(L_i\log x) \le 5\bggg(r_0 + \sum_{i \ge 1} r(Y_i)\bggg) \le 10\max\{\thetasmall, \thetalarge\}.
    \]
    It follows that
    \[
        \pahk[N_x \in \mathcal H(y, x^{T_x})] \le \pahk[N_x \in \mathcal H(y, e^{10\cdot \thetalarge})] + \pahk[M^*_x \in \mathcal H(y, e^{10\cdot \thetasmall})] + \pahk[\mathcal E(x)^c]
    \]
    with $\mathcal E(x) = \{N_x = M_x = M_x^*\}$. By combining Proposition \ref{prop:dtv of N M M*} with Lemmas \ref{lem:thetasmall H bound} and \ref{lem:thetalarge H bound}, we prove \eqref{eq:rewriting PDBL} and thus we finish the proof of Main Lemma \ref{lem:pd boundary lemma}.
\end{proof}

\end{document}